\documentclass[12pt]{amsart}
\usepackage{amsmath,amsthm,amsfonts,amssymb}
\usepackage{graphicx}

\newcommand{\ph}{\operatorname{\mathcal A}}
\newcommand{\diff}{\operatorname{Diff}}
\newcommand{\nuh}{\operatorname{Nuh}}

\newcommand{\N}{\mathbb N}
\newcommand{\PP}{\mathcal{P}}
\newcommand{\eps}{\varepsilon}
\newcommand{\transv}{\pitchfork}
\newcommand{\R}{\mathbb R}
\newcommand{\ess}{\text{ess}}
\newtheorem{theorem}{Theorem}[section]

\newtheorem{definition}[theorem]{Definition}
\newtheorem{proposition}[theorem]{Proposition}
\newtheorem{lemma}[theorem]{Lemma}
\newtheorem{conjecture}[theorem]{Conjecture}
\newtheorem{question}[theorem]{Question}
\newtheorem{corollary}[theorem]{Corollary}
\newtheorem{remark}[theorem]{Remark}

\title{Genericity of non-uniform hyperbolicity in dimension 3}
\author{Jana Rodriguez Hertz}
\begin{document}
\begin{abstract}
For a generic conservative diffeomorphism of a closed connected 3-manifold $M$, the Oseledets splitting is a globally dominated splitting. Moreover, either all Lyapunov exponents vanish almost everywhere, or else the system is non-uniformly hyperbolic and ergodic.\par
This is the 3-dimensional version of the well-known result by Ma\~{n}\'{e}-Bochi \cite{manhe1983, bochi2002}, stating that a generic conservative surface diffeomorphism is either Anosov or all Lyapunov exponents vanish almost everywhere. This result inspired and answers in the positive in dimension 3 a conjecture by Avila-Bochi \cite{avila_bochi2009}.
\end{abstract}
\maketitle
\section{Introduction}
Let $m$ be a smooth volume form on a closed connected Riemannian manifold $M$. Given a conservative diffeomorphism $f$ on $M$, the Lyapunov exponents evaluate  the exponential growth of the norm of the derivative along the direction of a given vector. Namely, for any tangent vector $v$ in $T_xM$,  the associated {\em Lyapunov exponent} is 
\begin{equation}\label{equation lyap exponent}
\lambda(x,v)=\lim_{n\to\pm\infty}\frac{1}{n}\log \|Tf^n(x)v\|
\end{equation}
This amount is not necessarily well defined; however, Oseledets \cite{oseledec} proved in the sixties that for $m$-almost every point in the manifold, there exists a splitting of the tangent bundle $$T_xM=E_1(x)\oplus\dots\oplus E_{l(x)}(x)$$ and numbers $\hat\lambda_1(x)>\dots>\hat\lambda_{l(x)}(x)$ such that for each vector in the subspace $E_i(x)$ its associated Lyapunov exponent is $\hat\lambda_i(x)$.\par
As it is expectable from its generality, the variation of the Oseledets splitting with respect to $x$ is only measurable. The limits in the formula (\ref{equation lyap exponent}) are not uniform, but they are also measurable with respect to $x$. \par
In 1983, Ma\~{n}\'{e} suggested that $C^1$-generically the Oseledets splitting and the Lyapunov exponents had a more regular behavior \cite{manhe1983}. He went further to propose a program.  \par
In \cite{manhe1983}, Ma\~n\'e announced that the generic area-preserving diffeomorphism of a compact two-dimensional manifold are either Anosov, or else all Lyapunov exponents vanish almost everywhere. A complete proof of this fact was given only almost 20 years later by Bochi in \cite{bochi2002}, see Theorem \ref{teo.manhe.bochi}. In particular, the following generic dichotomy holds: either all Lyapunov exponents vanish almost everywhere or the diffeomorphism is ergodic and non-uniformly hyperbolic (no Lyapunov exponent vanishes).\par
In this work, we obtain a result analogous to Ma\~n\'e-Bochi's in the 3-dimensional setting: 

\begin{theorem}\label{teo.jana}
A $C^1$-generic conservative diffeomorphism $f$ of a closed connected 3-manifold $M$, satisfies one of the following alternatives:
\begin{enumerate}
\item all Lyapunov exponents of $f$ vanish  almost everywhere, or
\item $f$ is {\em coarsely} partially hyperbolic , non-uniformly hyperbolic, and ergodic.
\end{enumerate}
\end{theorem}
A diffeomorphism $f$ is {\em coarsely partially hyperbolic} if the tangent bundle admits a dominated splitting of the form 
$$TM=E^{cs}\oplus E^{u}\qquad \text{or}\qquad TM=E^{s}\oplus E^{cu}$$
where $E^s$ and $E^{u}$ are hyperbolic, and $E^{cs}$ and $E^{cu}$ are volume hyperbolic, that is, volume contracting or volume expanding respectively. \par
Let us call $\nuh(f)$ the {\em Pesin region} of $f$,  the set where there is non-uniform hyperbolicity, namely:
$$\nuh(f)=\{x:\text{all Lyapunov exponents of $x$ are different from zero}\}$$
\par
A first generalization of Ma\~n\'e-Bochi result was obtained by Bochi and Viana in \cite{bochi_viana2005}: they prove that generically among conservative systems, the Oseledets splitting of almost every orbit is dominated, see more details in Theorem \ref{teo.bochi.viana}. An invariant splitting $T_{o(x)}M=E_{o(x)}\oplus F_{o(x)}$ is dominated over the orbit of $x$ if there is an integer $l\geq 1$ such that all unit vectors $v_E\in E_x$ and $v_F$ in $F_x$ satisfy:
$$\|Df^l(x)v_E\|\leq \frac12\|Df^l(x)v_F\|$$
This high-impact result, however, does not preclude the coexistence of different behaviors such as the vanishing of all Lyapunov exponents and a non-trivial Pesin region. In particular, even though it brings useful information in specific cases, no information is obtained about genericity of ergodicity, or the extension of the Oseledets splitting to a dominated splitting defined over the whole manifold  $M$. \par
More recently, Avila and Bochi \cite{avila_bochi2009} improved this result by showing that for a generic conservative diffeomorphism, either the Pesin region has measure zero, or else it is a dense ergodic component, and the Oseledets splitting extends to a dominated splitting over the manifold, see Theorem \ref{theorem avila bochi}. This is the first result about generic global dominancy of the Oseledets splitting since \cite{manhe1983, bochi2002}. \par
It could happen, {\em a priori}, however, that the Pesin region does not have total measure. 
Indeed, there could exist  a positive measure set of points with at least one zero Lyapunov exponent. 
Nevertheless, the existence of a global dominated splitting implies that there are at least two exponents different from zero amost everywhere, one that is positive and the other one that is negative. \par

As a by-product of Bochi-Viana's \cite{bochi_viana2005} and Avila-Bochi's \cite{avila_bochi2009} results, we obtain from Theorem \ref{teo.jana}:
\begin{corollary}
 $C^{1}$-generically among conservative diffeomorphisms of closed connected 3-manifolds, the finest Oseledets splitting is globally dominated. 
\end{corollary}
The core of the proof of Theorem \ref{teo.jana} is to show that if a diffeomorphism has a positive measure set where the three Lyapunov exponents take, respectively, positive, negative and zero signs, then owe can perturb it so that either this set vanishes, or the perturbation is partially hyperbolic on the whole manifold, see the sketch of the proof below (Subsection \ref{subsection.sketch.of.the.proof}). Since the set of partially hyperbolic diffeomorphisms is open, we use that ergodicity is generic among $C^1$-partially hyperbolic diffeomorphisms \cite{rhrhu2008}, so that we obtain a generic set of partially hyperbolic diffeomorphisms with constant Lyapunov exponents. The techniques of Baraviera-Bonatti \cite{baraviera_bonatti2003} then allow us to remove the zero Lyapunov exponent.\par
Theorem \ref{teo.jana} is the 3-dimensional version of Theorem \ref{teo.manhe.bochi} by Ma\~n\'e-Bochi. This theorem, and Theorem \ref{theorem avila bochi} by Avila-Bochi inspired them the following general conjecture:
\begin{conjecture}  [Avila-Bochi \cite{avila_bochi2009}]\label{conjecture.avila.bochi} For a $C^{1}$-generic diffeomorphism $f$ of a closed connected manifold, either one of the following holds:
\begin{enumerate}
  \item all Lyapunov exponents vanish almost everywhere, or else
  \item the Pesin region $\nuh(f)$ has full measure, $f$ is ergodic and the Oseledets splitting extends to a global dominated splitting.
\end{enumerate}
\end{conjecture}

Very recently, Avila, Crovisier and Wilkinson have announced a proof of this conjecture, by different methods from the ones presented here.\par
It is interesting to note that in dimension 2, the Ma\~n\'e-Bochi Theorem implies that for all surfaces other than the 2-torus, the generic situation is that all Lyapunov exponents vanish almost everywhere, whereas in the 2-torus, the Pesin region has full measure generically in the set of Anosov diffeomorphisms. However, in the 2-torus there is also a non-empty open set of diffeomorphisms inside which the generic system has vanishing Lyapunov exponents almost everywhere; this happens for instance, in the isotopy class of the identity.\par
In dimension 3, the following holds: for every 3-manifold there is a non-empty open set of conservative diffeomorphisms for  which generically all Lyapunov exponents vanish almost everywhere, as Grin showed in her Msc. thesis \cite{grin2010}. 

\begin{question}
  Is it true that generically in $\diff^1_m({\mathbb S}^3)$, all Lyapunov exponents vanish almost everywhere?
\end{question}
In \cite{avila_bochi2009} it is shown that even-dimensional spheres ${\mathbb S}^{2k}$ do not admit a globally dominated spitting. This implies that $C^{1}$-generically in ${\mathbb S}^{2k}$, the Pesin region has measure zero, as follows from Theorem \ref{theorem avila bochi} below. \par
The following is also unknown:
\begin{question}
  For which manifolds $M^3$ is there a non-empty open set ${\mathcal U}\subset\diff^1_m(M)$ such that generically in ${\mathcal U}$ the Pesin region has total measure?
\end{question}
 In Grin's thesis, it is also proven that for every $n$-manifold, there is an open set ${\mathcal U}\subset \diff^1_m(M)$, such that for a generic $f$ in ${\mathcal U}$, $m(\nuh(f))=0$. In particular the Pesin region can not have full measure for all diffeomorphisms in a generic subset of $\diff^{1}_{m}(M)$ on any manifold $M$. \newline\par
A key result in proving Theorem \ref{teo.jana} is the following:

\begin{theorem} \label{theorem.dificil} Let $f\in \diff^{r}_{m}(M)$, with $r>1$. For any partially hyperbolic compact invariant set $K$ such that $0<m(K)<1$ and $\dim E^{c}=1$, there are two periodic points $p$ and $q$ in $K$ that are {\em strongly heteroclinically related}, that is
 $$W^{ss}(p)\cap W^{uu}(q)\ne\emptyset\qquad\text{and}\qquad W^{uu}(p)\cap W^{ss}(q)\ne\emptyset$$
 and the intersections are {\em quasi-transverse}, that is, for each $z\in W^{ss}(p)\cap W^{uu}(q)$, the intersection space is trivial: $T_{z}W^{ss}(p)\cap T_{z}W^{uu}(q)=\{0\}$. An analogous statement holds for $z\in W^{uu}(p)\cap W^{ss}(q)$.
\end{theorem}
See more details in section 5. We emphasize that this is not a generic result, it holds for {\em every} diffeomorphism $f$ of a manifold of {\em any} dimension, as long as $\dim E^{c}=1$. As a consequence, we have
\begin{corollary}\label{teo.C.r.genericity.of.ph}
  For a generic $f$ in $\diff^r_m(M)$, with any $r\in[1, \infty]$, if there exists a compact partially hyperbolic set $K\subset M$ such that $m(K)>0$ and $\dim E^{c}=1$, 
  then $f$ is partially hyperbolic over $M$. Moreover,  if $K$ is $l$-partially hyperbolic, then $f$ is $l$-partially hyperbolic over $M$.
\end{corollary}
The precise meaning of partial hyperbolicity and $l$-partial hyperbolicity is given in Definition \ref{definition partially hyperbolic}.

\subsection{Sketch of the proof}\label{subsection.sketch.of.the.proof}
Let $M$ be a closed connected 3-dimensional manifold. For $f\in \diff^1_m(M)$, let $\lambda_1(x)\geq\lambda_2(x)\geq\lambda_3(x)$ be the Lyapunov exponents. Call ${\mathcal R}eg$ the set of {\em regular points} $x$ for which the limit (\ref{equation lyap exponent}) exists for every $v\in T_{x}M$. Recall that $m({\mathcal R}eg)=1$.
Define the following sets
\begin{equation}\label{equation.sets}
\begin{array}{l@{}c@{}lcl@{\hspace*{3pt}}l@{\hspace*{4pt}}c@{\hspace*{3pt}}r@{\hspace*{3pt}}r}
B&_{000}&(f)&=&\{x\in {\mathcal R}eg: \lambda_3(x)&=&\lambda_2(x)=0&=&\lambda_1(x)\}\\\noalign{\medskip}
B&_{--+}&(f)&=&\{x\in {\mathcal R}eg: \lambda_3(x)&\leq&\lambda_2(x)<0&<&\lambda_1(x)\}\\\noalign{\medskip}
B&_{-++}&(f)&=&\{x\in {\mathcal R}eg: \lambda_3(x)&<&0<\lambda_2(x)&\leq &\lambda_1(x)\}\\\noalign{\medskip}
B&_{-0+}&(f)&=&\{x\in {\mathcal R}eg:\lambda_3(x)&<&\lambda_2(x)=0&<&\lambda_1(x)\}
\end{array}
\end{equation}
The sets $B_\sigma(f)$, with $\sigma=000,--+,-++,-0+$ are invariant disjoint sets that form a partition of $M$ modulo a zero set. \par
Firstly, we show the main result, Theorem \ref{teo.jana}, under the hypothesis that generically in $\diff^{1}_{m}(M)$ the set  $B_{-0+}(f)$ has zero measure (Proposition \ref{proposition generic B} of Section \ref{section.particion}) .  
The rest of the paper consists in proving that, indeed, the hypothesis that generically in $\diff^1_m(M)$, the set $B_{-0+}(f)$ has  zero measure, is satisfied. \par
Now, generically in $\diff^{1}_{m}(M)$, if $m(B_{-0+}(f))>0$, there is a partially hyperbolic set with positive measure.  This essentially follows from Bochi-Viana's Theorem \ref{teo.bochi.viana}, see also Proposition \ref{proposition bochi viana}. The most delicate step is to show that generically, if there is such a partially hyperbolic set, then it must be the whole manifold (Corollary \ref{teo.C.r.genericity.of.ph}). 
In that case we would have that generically, if  $m(B_{-0+}(f))>0$, then $f$ is partially hyperbolic. 
But generic partially hyperbolic diffeomorphisms are ergodic, due to Hertz-Hertz-Ures' result \cite{rhrhu2008}. Using a technique by Baraviera and Bonatti \cite{baraviera_bonatti2003}, we remove the center zero Lyapunov exponent (Proposition \ref{proposition.mas.cero.menos.es.magro}), obtaining a $C^{1}$-open set where non-uniform hyperbolicity is generic. This proves that generically the set $B_{+0-}(f)$ has zero measure, and the main result follows. 
\par
So, the most delicate step, as we have mentioned, is to show Corollary \ref{teo.C.r.genericity.of.ph}, namely, that $C^{r}$-generically, if there is a partially hyperbolic set with positive measure, then this set is the whole manifold, for all $1\leq r\leq \infty$.  We sketch below the proof for $r>1$, the case $r=1$ follows easily, more details in Section \ref{section proof proposition} .  \par
Assume that there is such a partially hyperbolic set which is not the whole manifold, then we can show that this set contains an immersed surface foliated by lines tangent to the hyperbolic bundles $E_{1}(x)$ and $E_{3}(x)$ of the partially hyperbolic splitting (see Definition \ref{definition partially hyperbolic}).  We call these lines {\em strong leaves}. The surface looks more or less like the one depicted in Figure \ref{figure.boundary.leaf}, and in general is not compact. 
\begin{figure}[h]
 \includegraphics[height=5cm, width=10cm]{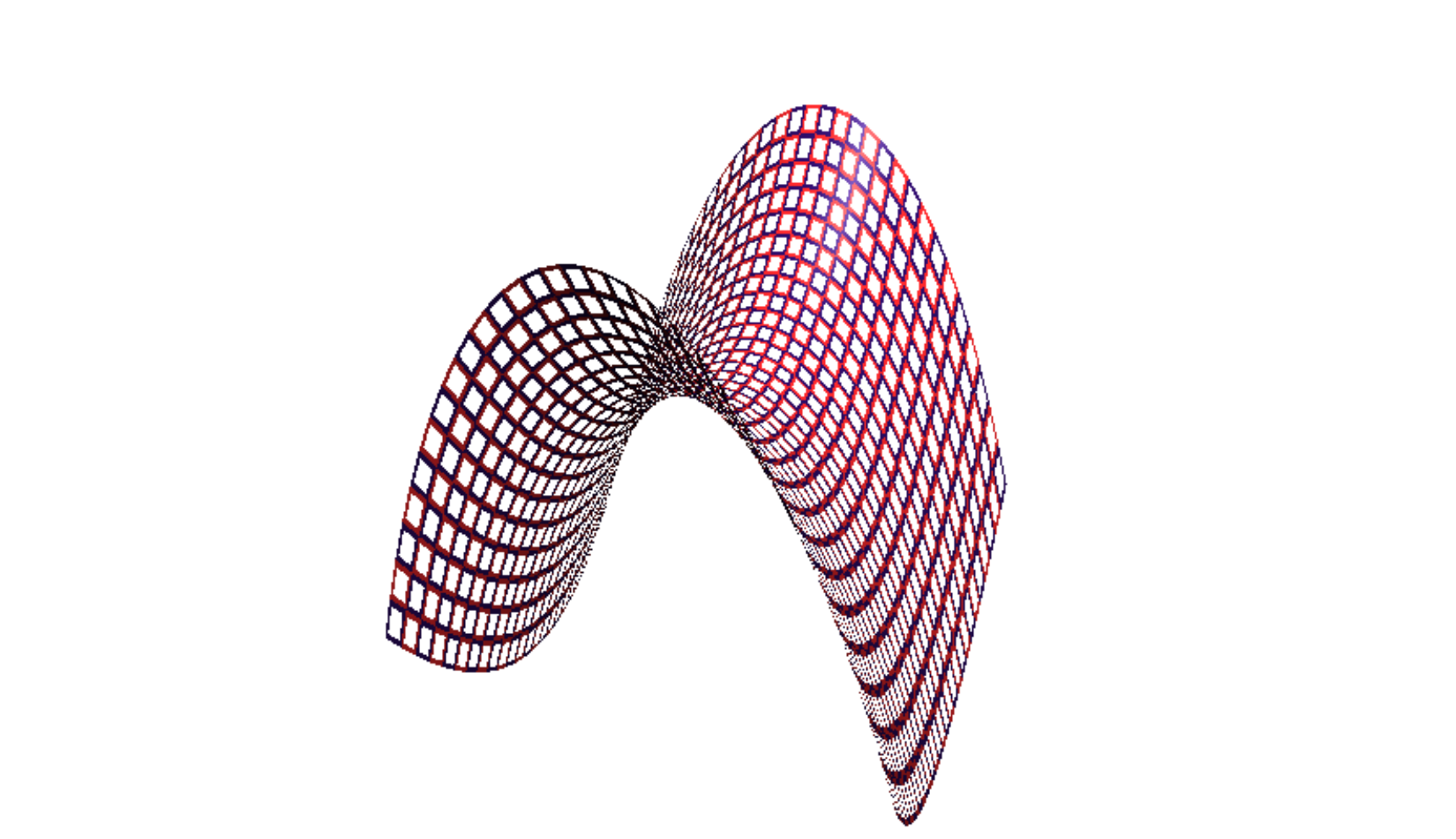}
\caption{ \label{figure.boundary.leaf} A surface foliated by strong leaves}
\end{figure}
The strong leaves meet at bounded from below angles. A Kupka-Smale type argument shows that the strong leaves of periodic points generically do not intersect. The rest, and most involved part of the proof consists in showing that periodic points accumulate on this surface with the intrinsic topology (Section  \ref{section proof of lemma final} ). This implies, due to the boundedness of the  angles, that their strong leaves do intersect, which is a non-generic situation. Therefore, generically this surface does not exist, so the partially hyperbolic set is either a zero measure set or the whole manifold.  
\subsection*{Acknowledgements.}I thank A. Avila, J. Bochi, F. Rodriguez Hertz, R. Ures and J. Yang for their enlightening conversations. I am grateful to the anonymous referee for her/his valuable comments. 
\section{Preliminaries}
Let $M$ be a closed connected Riemannian 3-manifold and let $m$ be a smooth volume measure. We consider the set $\diff^1_m(M)$ of $C^1$ diffeomorphisms preserving $m$, endowed with the $C^1$ topology.\par
As we have stated in the introduction, given a diffeomorphism $f\in\diff^1_m(M)$,  for $m$-almost every point there exists a splitting, the {\em Oseledets splitting}, $T_xM=E_1(x)\oplus\dots\oplus E_{l(x)}(x)$ and numbers $\hat\lambda_1(x)>\dots>\hat\lambda_{l(x)}(x)$, called the {\em Lyapunov exponents} such that for every non-zero vector $v\in E_i(x)$, we have
$$\lambda(x,v)=\hat\lambda_i(x),$$
where $\lambda(x,v)$, defined in Equation (\ref{equation lyap exponent}), is the exponential growth of the norm of $Tf$ along the direction of $v$.  The dimension of each $E_i(x)$ is called the {\em multiplicity of} $\hat\lambda_i(x)$. In our setting, we obtain, counting each Lyapunov exponent with its multiplicity: 
$$\lambda_1(x)\geq\lambda_2(x)\geq\lambda_3(x).$$

\begin{definition}[Dominated splitting] \label{definition dominated splitting} Given an $f$-invariant set $\Lambda$, and two invariant sub-bundles of $T_\Lambda M$: $E$ and $F$ such that $T_\Lambda M=E_\Lambda\oplus F_\Lambda$, we call the splitting an {\em $l$-dominated splitting} if for all $x\in\Lambda$ and all unit vectors $v_E\in E_x$ and $v_F\in F_x$ we have:
\begin{equation}\label{equation dominated splitting}
\|Df^l(x)v_F\|\leq \frac12\|Df^l(x)v_E\|
\end{equation}
We denote $E_\Lambda\succ_l F_\Lambda$
\end{definition}
Note that we do not require $\Lambda$ to be compact. In particular, we shall denote $E_x\succ_l F_x$ when the inequality (\ref{equation dominated splitting}) is satisfied for the orbit of $x$.
A splitting will be called {\em dominated} if there exists $l\in\N$ such that it is $l$-dominated. When a splitting is dominated, the direction of a vector not in $E$ nor in $F$ will converge to $E$ under forward iterates and to $F$ under backward iterates.
\begin{remark}\label{remark.Lyap.exponents.dom.splitting} For all $x\in {\mathcal R}eg$, if $E_x\succ F_x$,  $v_E\in E_x$ and $v_F\in F_x$, then the corresponding Lyapunov exponents satisfy $\lambda(x,v_E)>\lambda(x,v_F)$.
\end{remark}
\begin{definition}[Partial hyperbolicity]\label{definition partially hyperbolic} We say that $f$ is {\em $l$-partially hyperbolic} over an invariant set $\Lambda$ if the tangent bundle splits into three non-trivial invariant sub-bundles $E^s$, $E^c$, $E^u$, i.e. $T_\Lambda M=E^s\oplus E^c\oplus E^u$, and there exists $k\in\N$ such that for each $x\in\Lambda$
\begin{enumerate}
  \item $Df^{k}(x)$ is contracting on $E^s_x$
  \item $Df^{-k}(x)$ is contracting on $E^u_x$
  \item $E^u_\Lambda\succ_l E^c_\Lambda\succ_l E^s_\Lambda$
\end{enumerate}
$f$ is {\em partially hyperbolic} over $\Lambda$ if it is $l$-partially hyperbolic for some $l\in \N$.
\end{definition}
As in the definition of dominated splitting, $\Lambda$ need not be compact. If $f$ is partially hyperbolic over $\Lambda$, we also say that $\Lambda$ is a partially hyperbolic set. A diffeomorphism $f$ is {\em partially hyperbolic} if there is a partially hyperbolic splitting over the whole manifold $M$. Call $\mathcal{PH}^{1}_{m}(M)$ the set of partially hyperbolic diffeomorphisms in $\diff^{1}_{m}(M)$. \par
 As we mentioned in the Introduction, Ma\~n\'e and Bochi proved that there is a dichotomy for generic conservative diffeomorphisms on surfaces:
\begin{theorem}[Ma\~n\'e-Bochi \cite{manhe1983},\cite{bochi2002}] \label{teo.manhe.bochi}For a $C^{1}$-generic conservative diffeomorphism $f$ on a closed connected surface, either one of the following holds:
\begin{enumerate}
  \item all Lyapunov exponents vanish almost everywhere
  \item $f$ is Anosov
\end{enumerate}
\end{theorem}
This dichotomy was later generalized in some weaker sense by Bochi ad Viana:
\begin{theorem}[Bochi-Viana\cite{bochi_viana2005}] \label{teo.bochi.viana} For a generic $f\in \diff^1_m(M)$, the Oseledets splitting of almost every $x$ is dominated.
\end{theorem}
More recently, Avila and Bochi found a stronger formulation, that is valid for manifolds of any dimension. Call 
$$E^{+}(x)=\bigoplus_{\lambda>0} E_{\lambda}(x)\qquad E^{-}(x)=\bigoplus_{\lambda<0} E_{\lambda}(x)$$
The {\em zipped Oseledets splitting} of $x$ is defined as $T_{x}M=E^{+}(x)\oplus E_{0}(x)\oplus E^{-}(x)$. 
\begin{theorem}[Avila-Bochi\cite{avila_bochi2009}]\label{theorem avila bochi}
For a $C^{1}$-generic conservative diffeomorphism of a closed connected manifold, either:
\begin{enumerate}
\item the Pesin region has measure zero: $m(\nuh(f))=0$, or
\item $f|\nuh(f)$ is ergodic and $\nuh(f)$, is everywhere dense, that is, it meets every open set in a positive measure set. The zipped Oseledets splitting over $\nuh(f)$ 
extends to a global splitting $TM=E^+\oplus E^-$ that is dominated.
\end{enumerate}
\end{theorem}

\section{Lyapunov exponents and global dominated splitting}\label{section.particion}
From now on, let $M$ be closed connected Riemannian 3-manifold, and $f$ be a diffeomorphism in $\diff^1_m(M)$. Let $B_{000}(f)$, $B_{--+}(f)$, $B_{-++}(f)$ and $B_{-0+}(f)$ the invariant sets defined by Formula (\ref{equation.sets}), that form a partition of $M$, modulo a zero set.

\begin{proposition}\label{proposition generic B} Let $M$ be a closed connected manifold. 
For a generic diffeomorphism $f\in \diff^1_m(M)$, if $m(B_{-0+}(f))=0$, then $m(B_\sigma(f))=1$ for some $\sigma=000,--+,-++$.\newline\par
That is, for a generic diffeomorphism $f\in \diff^1_m(M)$, one and only one of the following holds:
\begin{enumerate}
  \item $m(B_{-0+}(f))>0$
  \item all Lyapunov exponents of $f$ vanish almost everywhere
  \item $f$ is non-uniformly hyperbolic, ergodic, and the Oseledets splitting extends to a globally dominated spitting on $M$.
\end{enumerate}
\end{proposition}
\begin{proof} Assume $m(B_{-0+}(f))=0$. 
If $m(B_{--+}(f)\cup B_{-++}(f))>0$, then the Pesin region $\nuh(f)=B_{--+}(f)\cup B_{-++}(f)$ is non-trivial and therefore, by Theorem \ref{theorem avila bochi} of Avila-Bochi, the zipped Oseledets splitting extends to a global dominated splitting $TM=E^+\oplus E^-$.\par
This implies that $m(B_{000}(f))=0$.  Indeed, since $f$ is conservative we always have $\lambda_1(x)\geq0\geq\lambda_3(x)$ with $\lambda_1(x)+\lambda_2(x)+\lambda_3(x)=0$. But the domination of the zipped Oseledets splitting implies that $\lambda_1(x)>\lambda_3(x)$ for $m$-a.e., as seen in Remark \ref{remark.Lyap.exponents.dom.splitting}, so $m(B_{000}(f))=0$.\par
Since, by hypothesis $m(B_{-0+}(f))=0$, and we have just proved that $m(B_{000}(f))=0$, then $f$ is non-uniformly hyperbolic on $M$, that is $m(\nuh(f))=1$. This implies by Theorem \ref{theorem avila bochi} that $f$ is ergodic. Since $B_{--+}(f)$ and $B_{-++}(f)$ are invariant sets, and their union is $M$ modulo zero, then one of them has full measure. And, since almost every orbit is dense, we have, by \cite{bochi_viana2005}, that the finest Oseledets splitting extends to a globally dominated splitting.\par
If, on the contrary, $m(B_{--+}(f)\cup B_{-++}(f))=0$, then we obviously have $m(B_{000}(f))=1$. \end{proof}
The rest of the paper consists in showing that generically we are in the hypotheses of Proposition \ref{proposition generic B}.
\section{Diffeomorphisms having a partially hyperbolic set}\label{section.difeos.having.a.ph.set}
Let us define the set $B_{str}(f)=\{x\in {\mathcal R}eg:\lambda_1(x)>\lambda_2(x)>\lambda_3(x)\}$, the invariant set of $x$ for which the inequalities between Lyapunov exponents are strict. Then $B_{+0-}(f)\subset B_{str}(f)$.
Let $$\mathcal{ A}(M)=\{f\in \diff^1_m(M):m(B_{str}(f))>0\}$$
Define the $l$-{\em partially hyperbolic set} $PH(f,l)$ of $f$ by 
\begin{equation}\label{ph(g,l)}
PH(f,l)=\overline{\{x\in {\mathcal R}eg: E_1(x)\succ_l E_2(x)\succ_l E_3(x) \}}
\end{equation}
This set is $l$-partially hyperbolic in the sense of Definition \ref{definition partially hyperbolic}.  Indeed, the $l$-domination and Remark \ref{remark.Lyap.exponents.dom.splitting} imply that for all $x\in PH(f,l)\cap {\mathcal R}eg$,  $\lambda_{1}(x)>\lambda_{2}(x)>\lambda_{3}(x)$.  Since $f$ is conservative $\lambda_{1}(x)+\lambda_{2}(x)+\lambda_{3}(x)=0$, this implies that $\lambda_{1}(x)>0$ and $\lambda_{3}(x)<0$ for a total measure set in $PH(f,l)$, that is $\mu$-almost every $x$ for each measure $\mu$ supported on $PH(f,l)$. Ma\~n\'e  proves in \cite[pages 521,522]{manhe1982} that this implies that $Df^{-k}(x)$ is contracting on $E_{1}(x)$ and $Df^{k}(x)$ is contracting on $E_{3}(x)$ for some $k\in \N$.\par 
 Conversely, all $l$-partially hyperbolic sets are contained in $PH(f,l)$ modulo a zero set.\par
 
For each $l,n\in {\mathbb N}$, we consider the sets $\PP^{r}_{l,n}(M)$ of $C^{r}$-diffeomorphims having an $l$-partially hyperbolic set with measure greater or equal than $1/n$, that is:
\begin{equation}
\PP^{r}_{l,n}(M)=\left\{f\in \diff^r_m(M): m(PH(f,l))\geq \frac{1}{n}\right\}
\end{equation}
Then we have:
\begin{proposition}\label{proposition bochi viana}
The following set is meager in $\diff^1_m(M)$:
\begin{equation}\label{equation conjunto M}
\ph(M)\setminus\bigcup_{l,n\geq 1}\PP^{1}_{l,n}(M)
\end{equation}
That is, for a generic diffeomorphism in $\diff^{1}_{m}(M)$, if $f$ is in ${\mathcal A}(M)$ then $f$ has an $l$-partially hyperbolic set with positive measure, for some $l\in\N$.
\end{proposition}
\begin{proof}
Call ${\mathcal R}$ the residual set of diffeomorphisms of $\diff^{1}_{m}(M)$ obtained in Theorem \ref{teo.bochi.viana} for which the Oseledets splitting of almost every orbit is dominated. 
 For $f\in\mathcal{ R}\cap\ph(M)$ and $m$-a.e. $x\in B_{str}(f)$, the Oseledets splitting of $x$, which is of the form $T_x M=E^1_x\oplus E^2_x\oplus E^3_x$, is $l$-dominated for some $l(x)\geq 1$.\par
In other words, for $f\in \mathcal{ R}\cap \ph(M)$, the set $B_{str}(f)$ coincides modulo a zero set with the set $\bigcup_{l=1}^\infty PH(f,l)$. By definition, if $f\in\mathcal{R}\cap\ph(M)$, we have $m(B_{str}(f)>0$, but being the union of sets $PH(f,l)$ countable, at least one of them will satisfy $m(PH(f,l))>0$, therefore $m(PH(f,l))\geq\frac{1}{n}$ for some $n$. From this we get that $\mathcal{ R}\cap \ph(M)$ is contained in $\bigcup_{l,n}\PP^{1}_{l,n}(M)$, and the claim follows.
\end{proof}

\begin{proposition}\label{lema.pln.closed}
For each $l,n\in \N$, and $r\geq 1$ the set $\PP^{r}_{l,n}(M)$ is $C^r$-closed.
\end{proposition}
\begin{proof}
Let $f_k$ be a sequence in $\PP^{r}_{l,n}(M)$ such that $f_k\to f$ in $\diff^r_{m}(M)$. Then, by definition, we have  $m(PH(f_k,l))\geq \frac{1}{n}$ for all $k\in \N$.\par
The sets $PH(f_k,l)$ are all compact. Let us consider $K$ the Hausdorff limit of any converging subsequence of $PH(f_{k}, l)$'s in the Hausdorff topology. Then $K$ is contained in $PH(f,l)$, since the $l$-dominated splittings vary continuously with respect to $f$ in the $C^r$-topology. \par
 On the other hand, $m$ is upper semi-continuous with respect to Hausdorff limits, that is, if $A_n$ converge to $A$ in the Hausdorff topology, then $\limsup_n m(A_n)\leq m(A)$. This implies that $m(PH(f,l))\geq \limsup_k m(PH(f_k,l))\geq \frac{1}n$, and proves the claim.
\end{proof}
As a consequence, we have that $\PP^{r}_{l,n}(M)\setminus{\rm int}(\PP^{r}_{l,n}(M))$ is a closed set with empty interior, and hence we obtain: 
\begin{corollary}\label{corollary.closure.interior}
 For each $r\in[1,\infty]$ and $l,n\in\N$, the following set is meager in $\diff^{r}_{m}(M)$
 $$\PP^{r}_{l,n}(M)\setminus \overline {{\rm int}(\PP^{r}_{l,n}(M))}$$
\end{corollary}
As a consequence of Corollary \ref{teo.C.r.genericity.of.ph} to be proved in Sections \ref{section proof proposition} and \ref{section proof of lemma final}, we have 
\begin{equation}
\mathcal{PH}^{1}_{m}(M)=\bigcup_{l,n}\overline{{\rm int}(\PP^{1}_{l,n}(M))}
\end{equation}
This, together with Corollary \ref{corollary.closure.interior} and Proposition \ref{proposition bochi viana} implies that the set 
$$\ph(M)\setminus {\mathcal PH}^{1}_{m}(M)\qquad \text{is meager}$$

Let us now define the set
$$\ph_0(M)=\{f\in \ph(M): m(B_{-0+}(f))>0\}.$$

\begin{proposition}\label{proposition.mas.cero.menos.es.magro}
$\ph_0(M)$ is meager in $\diff^1_m(M)$. That is, generically in $\diff^1_m(M)$, the set $B_{-0+}(f)$ has measure zero.
\end{proposition}
\begin{proof}
From the discussion above we have that
$$\ph_0(M)\setminus\mathcal{ PH}^1_m(M)$$
is meager in $\diff^{1}_{m}(M)$, let us see that $\ph_{0}(M)\cap \mathcal{PH}^{1}_{m}(M)$ is meager too. \par
To see that the generic diffeomorphism in $\mathcal{PH}^{1}_{m}(M)$ is ergodic, recall that ergodicity is a $G_{\delta}$-property in $\diff^{r}_{m}(M)$ for all $r\geq 0$, due to a result by Oxtoby-Ulam in the 40's \cite{oxtoby_ulam}. A recent result by Avila  \cite{avila2008} establishes that $\diff^{2}_{m}(M)$ is dense in $\diff^{1}_{m}(M)$. Hertz-Hertz-Ures \cite{rhrhu2008} prove that $\mathcal{PH}^{2}_{m}(M)$ contains a $C^{1}$-open and $C^{2}$-dense set of ergodic diffeomorphisms. Then, ergodicity is $C^{1}$-generic among partially hyperbolic diffeomorphisms.\par
In other words, if we call $\mathcal{EPH}^1_m(M)$ the set of ergodic partially hyperbolic diffeomorphisms in $\diff^1_m(M)$,  the following set is meager
\begin{equation}\label{formula.ph}
  \ph_0(M)\setminus\mathcal{EPH}^1_m(M), 
\end{equation}
Due to Baraviera and Bonatti \cite{baraviera_bonatti2003}, there is a $C^1$-open and dense set of partially hyperbolic $f$ for which $\int\lambda_2(x,f)dm\ne0$. So, $\mathcal{EPH}^1_m(M)\cap\ph_0(M)$ is meager, and hence $\ph_0(M)$ is too.
\end{proof}

\section{Generic non-existence of proper partially hyperbolic subsets} \label{section proof proposition}

In this section we deduce Corollary \ref{teo.C.r.genericity.of.ph}, from Theorem \ref{theorem.dificil}
Assume first $r>1$. Let $K$ be a partially hyperbolic set such that $0<m(K)<1$. For each $x\in K$ there are immersed manifolds $W^{uu}(x)$ tangent to $E_{u}(x)$ and $W^{ss}(x)$ tangent to $E_{s}(x)$, see for instance \cite{HPS}. They are called, respectively, the {\em strong unstable} and {\em strong stable} manifolds. They are uniquely defined by the fact of being tangent to $E_{u}(x)$, respectively $E_{s}(x)$ at every point. In our case, they are immersed lines. \par
 From Theorem \ref{theorem.dificil} it follows there are two periodic points $p$  and $q$ in $K$ that are {\em strongly heteroclinically related}, that is 
$$W^{ss}(p)\cap W^{uu}(q)\ne \emptyset\qquad\text{and}\qquad W^{uu}(p)\cap W^{ss}(q)\ne \emptyset$$
where the intersection is {\em quasi-transverse}, that is $T_{x}W^{ss}(p)\cap T_{x} W^{uu}(q)=\{0\}$ for any $x\in W^{ss}(p)\cap W^{uu}(q)$.\par 
Two immersed manifolds $W$ and $V$ are {\em transverse} if either they do not intersect, or else they satisfy $$T_{x}W\oplus T_{x}V=T_{x}M$$
at every point $x\in V\cap W$. We denote $V\transv W$ if $V$ and $W$ are transverse immersed manifolds.
\begin{proposition}\label{proposition kupka-smale} For all $r\in[1,\infty]$, there exist a residual set of diffeomorphisms $\mathcal R$ such that for $f\in \mathcal R$, all $f$-periodic points $p, q$ are hyperbolic, and satisfy
$$W^{uu}(p)\pitchfork W^{ss}(q)\qquad\forall p\ne q$$
where $W^{uu}(p)$ and $W^{ss}(q)$ are the strong unstable of $p$ and the strong stable manifold of $q$, respectively.
\end{proposition}

The proof of Proposition \ref{proposition kupka-smale} can be straightforwardly adapted from the well known Kupka-Smale Theorem's proof. \par

Theorem \ref{theorem.dificil} and Proposition \ref{proposition kupka-smale}  imply that $C^{r}$-generically with $r>1$, there is no partially hyperbolic subset $K$ with $0<m(K)<1$. For if there were, there would be two periodic points, the strong stable and unstable manifolfds of which have non-trivial intersection. This intersection cannot be transverse, since the sum of the dimensions of the manifolds involved is $(n-1)$. \par
The case $r=1$ follows from the case $r>1$. Consider the $C^{1}$-interior of the set of diffeomorphisms $\PP^{1}_{l,n}(M)$. Due to Avila's regularization result \cite{avila2008}, the $C^{r}$-interior of $\PP^{r}_{l,n}(M)$ is $C^{1}$-dense in the interior of $\PP^{1}_{l,n}(M)$, for $r>1$. 
Since it was already established the validity of Corollary \ref{teo.C.r.genericity.of.ph} for $r>1$, we have that there is a $C^{1}$-dense set of $l$-partially hyperbolic diffeomorphisms in the interior of $\PP^{1}_{l,n}(M)$.\par
The set of $l$-partially hyperbolic diffeomorphisms is $C^{1}$-closed, so all $f$ in $\overline{\text{int}\PP^{1}_{l,n}(M)}$ are $l$-partially hyperbolic .  As we have seen in the previous section, the set of diffeomorphisms $f$ for which $m(PH(f,l))>0$ differs from the union over $n$ of ${\rm int}(\PP_{l,n}(M))$ in a meager set, so we get Corollary \ref{teo.C.r.genericity.of.ph} for $r=1$. 

\section{Proof of Theorem \ref{theorem.dificil}}\label{section proof of lemma final}
Let $f$ be a diffeomorphism in $\diff^{r}_{m}(M)$, with $r>1$. $M$ is a closed connected $n$-manifold.  
Let $K$ be a compact invariant $l$-partially hyperbolic set such that $0<m(K)<1$ and such that $\dim E^{c}_{x}=1$ for all $x\in K$. \par
The {\em essential closure} of $K$ is 
\begin{equation}\label{equation.ess.closure.PH}
\text{ess}(K)=\{x\in K: m(B_\eps(x)\cap K)>0\quad \forall \eps>0\}.
\end{equation}

$\text{ess}(K)$ is a closed subset of $K$ that contains all Lebesgue density points of $K$; therefore $m(\text{ess}(K))=m(K)$.  \par
A proof of the following lemma can be also found in \cite[Corollary B]{zhang}
\begin{lemma}\label{lemma.K(f,l).su.saturado}
If $\text{ess}(K)=K$ and $m(K)>0$, then for each $x\in K$
 $$W^{ss}(x)\cup W^{uu}(x)\subset K$$
\end{lemma}
\begin{proof}
For each $x\in M$ the {\em Pesin manifolds}, are defined by
$$W^{+}(x)=\left\{y\in M: \lim\sup\frac1n\log d(f^{-n}(x),f^{-n}(y))<0\right\}$$
and 
$$W^{-}(x)=\left\{y\in M: \lim\sup\frac1n\log d(f^{n}(x),f^{n}(y))<0\right\}$$
The Pesin manifolds are immersed manifolds, see for instance \cite{pesin}. 
The dimension of $W^{+}(x)$ equals the number of positive Lyapunov exponents, and analogously, the dimension of $W^{-}(x)$ equals the number of negative Lyapunov exponents. 
For $x\in K$ it is easy to check that 
\begin{equation} \label{equation.invariant.manifold}
W^{ss}(x)\subset W^{-}(x)\qquad\text{and} \qquad W^{uu}(x)\subset W^{+}(x)
\end{equation}
 If $K$ is any invariant set with positive measure, then for almost every point $x$ in $K$, almost every $y$ in $W^{+}(x)$ and almost every $z$ in $W^{-}(x)$ belong to $K$. A proof of this known fact may be found, for instance, in \cite[Lemma 4.3]{rhrhtu2009}. \par
 In our case, $K$ is compact, and all points in $K$ are accumulated by positive measure sets of points in $K$; hence $W^{+}(x)\cup W^{-}(x)\subset K$ for all $x\in K$. From (\ref{equation.invariant.manifold}) we get the claim.
\end{proof}
From now on, let us assume that $K=\ess(K)$.  For each $x\in K$, the {\em accessibility class} of $x$ is the set of points that can be joined to $x$ by means of a path that is piecewise tangent to $E^{s}$ or $E^{u}$; that is, the accessibility class is the minimal $s$- and $u$-saturated set that contains $x$. We denote it by $AC(x)$. It follows from Lemma \ref{lemma.K(f,l).su.saturado} that $AC(x)\subset K$. Denote by $O_{K}$ the set of open accessibility classes in $K$. Then 
\begin{proposition}\label{proposition.gamma.laminated}
The accessibility class of  each point in $\Lambda=K\setminus O_{K}$ is an immersed surface that is complete with the intrinsic topology. The angles between $s$- and $u$- leaves are uniformly bounded from below in the compact invariant set $\Lambda$. Moreover, the set $\Lambda$ is a compact lamination. 
\end{proposition}

\begin{proof}
As a first step, let us see that a point $x$ is in $O_{K}$ if and only if its accessibility class $AC(x)$ has non-empty interior. Take an open subset $U$ of $AC(x)$, and consider any point $y$ in $U$, and $z\in AC(x)$. Then, by definition, there is an $su$-path from $y$ to $z$, that is, a finite concatenation $s_{1}, u_{1},\dots, s_{k}, u_{k}$ of $s$- and $u$- leaves, having one end-point in $y$ and the other one in $z$. Take the set of all $s$- leaves through points in $U$. This is an open set, due to the continuity of the holonomies. Call this set $U_{1}$. Now consider all $u$-leaves through points in $U_{1}$. This is also an open set, contained in $AC(x)$. Defining inductively the open sets $U_{i}$, and considering successively the corresponding $s$- and $u$-leaves, we obtain that $z$ belongs to $U_{2k}$, which is an open set contained in the accessibility class of $x$, $AC(x)$. This proves that all points in $AC(x)$ are interior points, hence $x$ is in $O_{K}$, the set of points with open accessibility classes. See also \cite{rhrhu2008}.\par
 Now, let $x$ be a point in $\Lambda$, and consider the (closed) local stable manifold $W^{ss}_{\eps}(x)\subset W^{ss}(x)$. For each $y$ in the local stable manifold of $x$, take its (closed) local unstable manifold $W^{uu}_{\eps}(y)\subset W^{uu}(y)$. We call this set $W^{su}_{\eps}(x)$. $W^{su}_{\eps}(x)$ is contained in $AC(x)$ and is a topological $(n-1)$-manifold that separates a small ball around $x$ in two connected components.  To see this, recall that due to the Stable Manifold Theorem \cite{HPS}, $W^{ss}_{\eps}(x)$ is an embedding $\psi$ of a closed $s$-dimensional disc of radius 1, and for each $y\in W^{ss}_{\eps}(x)$, the set $W^{uu}_{\eps}(y)$ is an embedding $\phi_{y}$ of a closed $u$-dimensional disc of radius 1, the embeddings vary in a H\"older continuous way with respect to $y$. 
 The map $h_{x}:D^{s}\times D^{u}\to M$ such that $h_{x}(s,u)=\phi_{\psi(s)}(u)$ is continuous and injective. Since $D^{s}\times D^{u}$ is compact, $h_{x}$ is a homeomorphism onto its image. \par
Since $z\mapsto h_{z}$ is continuous, there exists $\delta>0$ such that for all $z\in B_{\delta}(x)$, the set $W^{su}_{\eps}(z)$ separates $B_{\delta}(x)$ in (at least) two connected components. Indeed, take a $C^{1}$ $(s+1)$-manifold $W^{s+1}(x)$ containing $W^{s}(x)$ and transverse to the leaves $W^{u}$. Extend $W^{u}$ to a continuous foliation in a small ball $B_{\delta}(x)$. Define a homeomorphism $\varphi:\R^{s}\times\R\times\R^{u}\to B_{\delta}(x)$ so that $\R^{s}\times \R$ parametrizes $W^{s+1}(x)$,  and $\varphi(0,0,0)=x$, in the natural way, that is, so that the image for each fixed $(s_{0},c_{0})$, 
$\varphi(s_{0},c_{0}, \R)$ is  $W^{u}(\varphi(s_{0},c_{0},0))$. \par
Since $\varphi$ is a homeomorphism, and $\R^{s}\times\{0\}\times\R^{u}$ separates $\R^{n}$, we get that $\varphi(\R^{s}\times\{0\}\times\R^{u})=W^{su}_{\eps}(x)$ separates $B_{\delta}(x)$.\par
\begin{figure}[h]
 \includegraphics[width=0.8\textwidth]{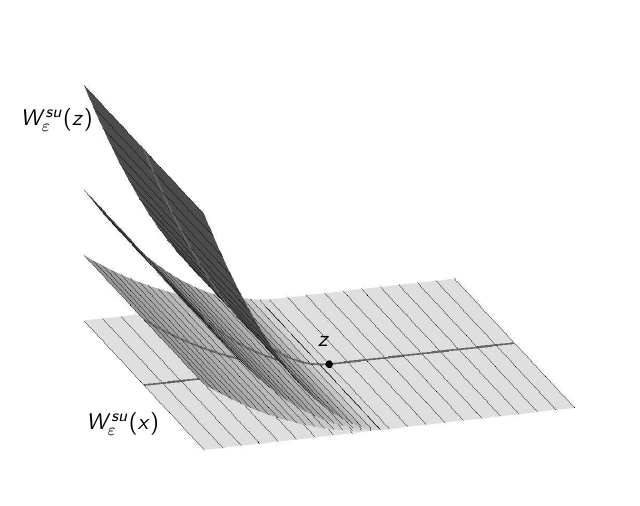}
 \caption{\label{figure.gamma.laminada} A point $z$ such that $W^{ss}_{\eps}(z)$ (the dark line) is not contained in $W^{su}_{\eps}(x)$}
\end{figure}
 Let us see that $W^{su}_{\eps}(x)$ contains all its local stable leaves in a small neighborhood of $x$. Take $z\in B_{\delta}(x)$ in the local unstable manifold $W^{uu}_{\eps}(y)$, where $y\in B_{\delta}(x)$ belongs to the local stable leaf of $x$, and consider its local stable manifold $W^{ss}_{\eps}(z)$. We claim that $W^{ss}_{\eps}(z)\cap B_{\delta}(x)$ is contained in the connected set $W^{su}_{\eps}(x)$. 
 If it were not the case, the surface $W^{su}_{\eps}(z)$ would be as in 
 Figure \ref{figure.gamma.laminada}, and it would cut a local center leaf of $x$, $W^{c}_{\eps}(x)$\footnote{Due to Peano, there exist local integral curves to the line field $E^{c}$. These curves may be non-unique, we call them local center curves.} in (at least) a point different from $x$, call it $w$. We may assume that the segment $[x,w]^{c}\subset W^{c}_{\eps}(x)$ is contained in $B_{\delta} (x)$. The two $(n-1)$ manifolds $W^{su}_{\eps}(x)$ and $W^{su}_{\eps}(z)$ contain the local unstable manifold $W^{uu}_{\eps}(y)$, and each of them separates the ball $B_{\delta}(x)$.  For each $\xi\in[x,w]^{c}$ 
 the set $W^{su}_{\eps}(\xi)$ separates $B_{\delta}(x)$, and hence intersects $W^{su}_{\eps}(x)$. Therefore $W^{su}_{\eps}(\xi)\subset AC(x)$ for each $\xi\in [x,w]^{c}$. The continuity of $\xi\mapsto h_{\xi}$ implies that $AC(x)$ has non-empty interior and, as we have show above,  $AC(x)$ is open. The point $x$ would not be in $\Lambda$, a contradiction.\par
 The paragraph above shows that the patches $W^{su}_{\eps}(x)$ have as a coordinate system the local stable and unstable leaves. The immersion of the patches in the manifold are $C^{r}$ when restricted along these coordinates. Then a lemma by Journ\'e \cite{journe} \footnote{The hypothesis $r>1$ is not being used in this part of the proof. This argument also holds for $r=1$} implies that the patches are $C^{1}$. The angles between the local stable and unstable leaves are bounded, due to the compactness of $K$. This implies that the size of the patches $W^{su}_{\eps}(x)$ are uniformly bounded from below. All these arguments together yield that $AC(x)$ is a complete immersed surface, for each $x\in \Lambda$. The continuity of the stable and unstable leaves implies that $\Lambda$ is, in fact, a lamination.
\end{proof}

There are two possibilities for the compact lamination $\Lambda=K\setminus O_{K}$:
\begin{enumerate}
 \item $\Lambda$ contains a compact leaf
 \item $\Lambda$ contains no compact leaves.
\end{enumerate}
Let us treat them separately:
\subsection{The lamination $\Lambda$ contains a compact leaf}
Consider the set of all compact leaves of $\Lambda$. H\"afliger proves in \cite{haefliger} that this set is also a compact lamination, call it $\mathcal T$. This lamination is also invariant. All leaves in $\mathcal T$ are compact manifolds subfoliated by $W^{ss}$ and $W^{uu}$. \par

\begin{definition}\label{definition.boundary.leaf}
Given a compact lamination $\Lambda$ contained in a partially hyperbolic set, a leaf $L\in\Lambda$ is called a {\em boundary leaf} if there exists a  local center segment $[x,w]^{c}\subset W^{c}(x)$, such that $x\in L$ and $(x,w]^{c}\cap\Lambda=\emptyset$.
\end{definition}

\begin{lemma}\label{lemma.periodic.tori}
 If $\mathcal T$ is the compact invariant lamination consisting of the compact leaves of $\Lambda$, then there is a boundary leaf of $\mathcal T$ that is periodic. 
\end{lemma}
\begin{proof}
If $\mathcal T$ has only a finite number of boundary leaves, then the result follows, since boundary leaves go into boundary leaves under the action of $f$. \par
 Suppose $\mathcal T$ has infinitely many boundary leaves. This implies there exists two boundary leaves $L_{1}$ and $L_{2}$ such that $L_{i}\subset B_{\eps}(L_{j})$ for $i,j=1,2$ and arbitrarily small $\eps>0$. They are such that the open set $V$ in between $L_{1}$ and $L_{2}$ does not intersect $\mathcal T$. Since $f$ is conservative, there exists $n>0$ such that $f^{n}(V)\cap V\ne\emptyset$.  For this $n>0$, we have in fact that $f^{n}(V)=V$, for otherwise $V\cap{\mathcal T}\ne\emptyset$. This implies that $L_{1}$ and $_{2}$ are invariant under $f^{2n}$ . 
\end{proof}
As a consequence of Lemma \ref{lemma.periodic.tori},  there exists a periodic manifold $L$ everywhere tangent to $E^{ss}\oplus E^{uu}$. It follows that $f|_{L}$ is Anosov. The non-wandering set of $f|_{L}$ contains infinitely many periodic points. Since $L$ is compact, there are arbitrarily close periodic points that are heteroclinically related inside $L$, that is there are $p,q\in L\subset K$ such that 
$$W^{ss}(q)\cap W^{uu}(p)\ne\emptyset\qquad\text{and}\qquad W^{uu}(q)\cap W^{ss}(p)\ne \emptyset$$
and the intersection is quasi-transverse. 
 This proves Theorem \ref{theorem.dificil} in the case that $\Lambda$ contains a compact leaf. 
\subsection{No leaf of $\Lambda$ is compact} Any connected component of $M\setminus \Lambda$ is called a {\em complementary region}.  $f$ takes complementary regions into complementary regions. Since $f$ is conservative, each complementary region is periodic. For simplicity, let us assume there is a complementary region and a boundary leaf $\Gamma$ that is fixed. \par
A {\em closed complementary region} is a  complementary region completed with respect to the induced path metric or, equivalently, a complementary region together with its boundary leaves. 
 If $\Lambda$ does not contain compact leaves, then each closed complementary region decomposes
into a compact piece, called {\em gut}, and $I$-bundles \footnote{An $I$-bundle is a fiber bundle whose fiber is an interval and whose base is a manifold. In our case the interval is compact.} over a non-compact manifold, called {\em interstitial regions}.  One can take the interstitial regions to be as thin as one wishes. They get thinner and thinner as they go away
from the gut. The decomposition into interstitial regions and guts is unique up to isotopy  (see \cite{gabai_kazez, hector_hirsch, rhrhu2008_2}). See Figure \ref{figure gut}. \par
\begin{figure}[h]
\includegraphics[width=0.5\textwidth]{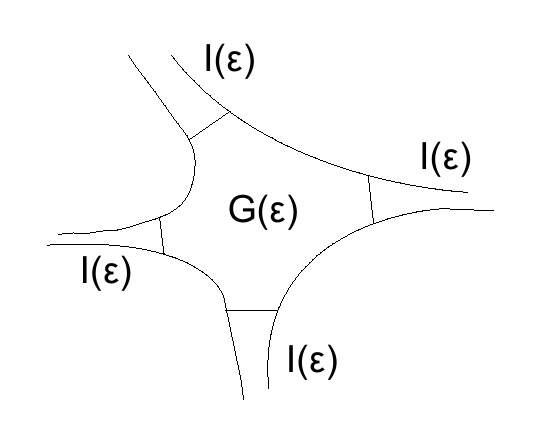} 
\caption{\label{figure gut} Decomposition into gut and interstitial regions}
\end{figure}

The guts and interstitial regions are obtained in the following way. Fix a small $\eps>0$. For each $x$ in a complementary region, consider a foliation box of the form $B_{x}=(-\eps,\eps)\times D$, where $(-\eps,\eps)$ is a segment transverse to the lamination, and $D$ is an $(n-1)$-dimensional disk. The coordinates are chosen so that $x=(0,0)$. The component of $B_{x}\setminus \Lambda$ containing $x$ is of the form $(-a,b)\times D$, with $0< a,b\leq \eps$. The interstitial regions consist of the points such that $a<\eps$ and $b<\eps$, completed with respect to the induced path metric. This is not a compact set. The completion of the points such that either $a=\eps$ or $b=\eps$ is the gut. \par
For each given complementary region, we shall call $I(\eps)$ the union of interstitial regions, and $G(\eps)$ the gut. They intersect in a compact set with empty interior. As it follows from the discussion above, $I(\eps)$ can be made as thin as we wish, by taking sufficiently small $\eps$. More details can be found in \cite{gabai_kazez, hector_hirsch}.

\begin{definition} Let us fix once and for all a boundary leaf $\Gamma$, considered with its intrinsic topology. $\Gamma$ is complete with this topology. For each $N>0$, define
$$I_{N}(\eps)=\{x\in\Gamma: f^{n}(x)\in I(\eps)\quad\forall n\geq N\}$$
 and define $$G_{*}(\eps)=\overline{\Gamma\setminus\bigcup_{N}I_{N}(\eps)}$$
 that is, $G_{*}(\eps)$ is the closure of the points that return infinitely many times to $G(\eps)$. 
\end{definition}
$\Gamma$ is a non-trivial complete set, and it is the countable union of the closed sets $I_{N}(\eps)$ with $N>0$ and $G_{*}(\eps)$ defined above. Hence, by a Baire category argument,  there is at least one these sets that has non-empty interior. \par
Next, we examine each of the possibilities and conclude that there must exist a pair of strongly heteroclinically related periodic points, which ends the proof of Theorem \ref{theorem.dificil}, and then of Theorem \ref{teo.jana}.

\subsubsection{$G_{*}(\eps)$ has non-empty interior:} $G_{*}(\eps)$ is a closed invariant set. By definition of $G_{*}(\eps)$ and compactness of $G(\eps)$, we have  
$\emptyset\ne\omega(x)\cap G(\eps)$ for all $x\in G_{*}(\eps)$. The dynamics of $f|_{G(\eps)}$ is hyperbolic, therefore $\omega(x)$ is the orbit of a periodic point $p$ if and only if $x\in W^{ss}(p)$.  Obviously, $W^{ss}(p)$ need not be contained in $G(\eps)$. Recall that $f|_{\Gamma}$ is hyperbolic. \par
Since $G_{*}(\eps)$ has non-empty interior, there exists $x\in G_{*}(\eps)$ such that $\omega(x)$ is not a periodic point, for otherwise $G_{*}(\eps)$ would be contained in a countable union of stable and unstable manifolds. Take a non-periodic point $y\in\omega(x)\cap G(\eps)$. The Anosov Closing Lemma applies and implies that there is a sequence of periodic points $p_{n}\to y$. Since the convergence is in the intrinsic topology, this implies that there are $p_{n_{1}}$ and $p_{n_{2}}$ in $\Gamma\subset K$ that are strongly heteroclinically related.

\subsubsection{An $I_{N}(\eps)$ has non-empty interior:} 
A {\em pre-lamination} over an invariant subset $\Gamma$ of  $K$ is a continuous choice of $C^{1}$ 1-dimensional discs $V^{c}(x)$ embedded in $M$ through each $x\in\Gamma$. That is, for each $x\in\Gamma$ there exists a $C^{1}$ embedding $V^{c}(x): E^{c}_{x}\cap B_{1}(0)\to M$; the embeddings $x\mapsto V^{c}(x)$ vary continuously, this means that if $x$ and $x'$ are close, then $V^{c}(x)$ and $V^{c}(x')$ are $C^{1}$-close. A pre-lamination $V^{c}(x)$ is {\em locally invariant} if for each $\rho>0$ there exists $\eps>0$ such that if $V^{c}_{\eps}(x)=V^{c}(x)|_{B_{\eps}(0)}$ then $f(V^{c}_{\eps}(x))\subset V^{c}_{\rho}(f(x))$ for all $x\in \Gamma$.

\begin{proposition}\label{proposition.HPS} \cite[Theorem 5.5]{HPS} For any invariant set $\Gamma\subset K$ there is a locally invariant pre-lamination $V^{c}(x)$ over $\Gamma$ such that each $V^{c}(x)$ is tangent to $E^{c}_{x}$ at $x$. 
\end{proposition}

Note that the locally invariant pre-lamination of Proposition \ref{proposition.HPS} is not necessarily made of local center curves. Namely, $V^{c}(x)$ is tangent to $E^{c}(x)$ at $x$, but not necessarily we have that $T_{y}V^{c}(x)=E^{c}(y)$ for other points $y\in V^{c}(x)$. In the sequel we assume that $\eps>0$ has been taken to satisfy the locally invariance property $f(V^{c}_{\eps}(x))\subset V^{c}_{\rho}(f(x))$ for some small $\rho>0$ such that $V^{c}_{\rho}(x)\cap I(\eps)\subset V^{c}_{\eps}(x)$ for all $x\in I(\eps)$. \par
For every $x\in I(\eps)\cap \Gamma$, $V^{c}(x)$ intersects $\Lambda$ at least twice, see Figure \ref{figura abierto I(eps)}. Let us see that if $\eps>0$ is small enough, for every $x_{0}$ in the interior of $I(\eps)\cap\Gamma$ the union of the $V^{c}(x)$ through a small plaque $W^{su}_{\eta}(x_{0})$ contains a small neighborhood of $x_{0}$ in $M$.  
\begin{figure}[h]
 \includegraphics[width=0.7\textwidth]{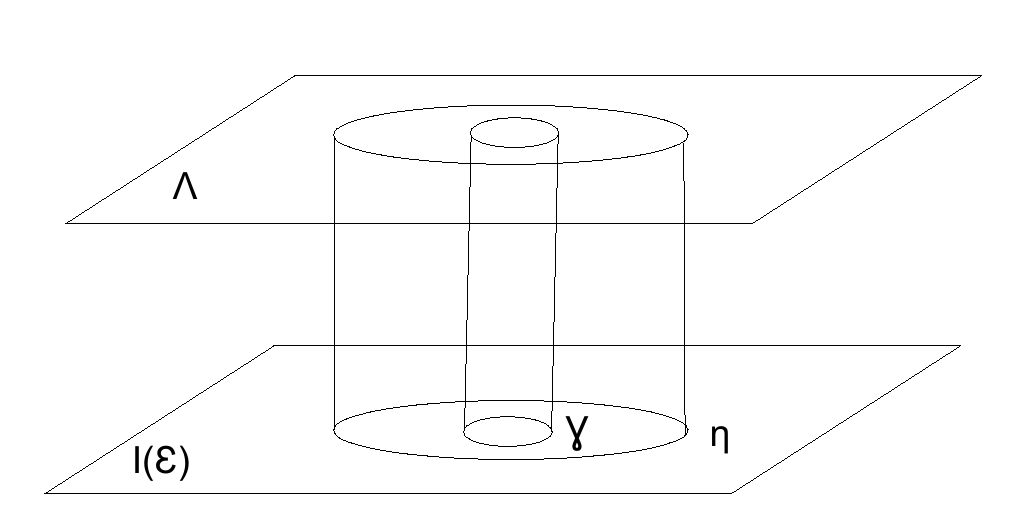}
 \caption{\label{figura abierto I(eps)} An open neighborhood of $x_{0}$ made of $V^{c}(x)$, with $x\in W^{su}_{\eta}(x_{0})$}
\end{figure}
\par
Given $x_{0}\in I(\eps)\cap \Gamma$, take $\eta>0$ such that for all $x\in \partial W^{su}_{\eta}(x_{0})$, we have $V^{c}(x)\cap V^{c}(x_{0})=\emptyset$. This is possible if $\eps>0$ is small, due to the transversality of $E^{c}(x_{0})$ and $E^{ss}\oplus E^{uu}$, (recall that, due to compactness of $K$, the angles are uniformly bounded from below). In fact, we may take a small $\gamma>0$ such that all $x$ as before, satisfy $V^{c}(x)\cap B_{\gamma}(V^{c}(x_{0}))=\emptyset$, see Figure \ref{figura abierto I(eps)}. For each $x\in W^{su}_{\eta}(x_{0})$, call $\Psi(x)$ the first intersection of $V^{c}(x)$ with $\Lambda$ that is not in $I(\eps)$. $\Psi$ is continuous; therefore,  $\Psi(\partial W^{su}_{\eta}(x_{0}))$ is homotopic to a point in $\Lambda$.  Also $\Psi(\partial W^{su}_{\eta}(x_{0}))$ does not intersect $B$, the connected component of $B_{\gamma}(\Psi(x_{0}))\cap \Lambda$ containing $\Psi(x_{0})$, by our choice of $\gamma$. Hence $\Psi(W^{su}_{\eta}(x_{0}))$ contains $B$, see Figure \ref{figura abierto I(eps)}.  This implies that the union of all $V^{c}(x)$, with $x\in W^{su}_{\eta}(x_{0})$ contains a small neighborhood of $x_{0}$ in $M$ with the ambient topology.\par

Assume that there is $N>0$ such that the interior of $I_{N}(\eps)$ is non-empty.  Consider a small $su$-plaque $W^{su}_{\eta}(x_{0})$ contained in $f^{N}(I_{N}(\eps))$, satisfying the conditions above. The union of all $V^{c}(x)$ with $x\in W^{su}_{\eta}(x_{0})$ contains an open neighborhood $W$ of $x_{0}$ in $M$. Note that, by our choice of $\eps>0$, we have $f(V^{c}_{\eps}(x))\subset V^{c}_{\rho}(f(x))\cap I(\eps)\subset V^{c}_{\eps}(f(x))$, for all $x\in I(\eps)\cap \Gamma$. By induction, we have for every $x\in f^{N}(I_{N}(\eps))$, for all $n\geq 0$
\begin{equation}\label{equation local invariance}
f^{n}(V^{c}_{\eps}(x))\subset V^{c}_{\eps}(f^{n}(x))
\end{equation}
Since $f$ is conservative, there is a recurrent point $y$ in $W$. By definition of $W$, $y\in V^{c}(y_{*})$, with $y_{*}\in W^{su}(x_{0})$. The point $y_{*}$ is recurrent too. Indeed, for arbitrarily large $n>0$, $f^{n}(y)\in W$, and by definition of $W$, $f^{n}(y)\in V^{c}(y^{n}_{*})$, with $y^{n}_{*}\in W^{su}(x_{0})$. The local invariance property (\ref{equation local invariance}) implies that  $y^{n}_{*}=f^{n}(V^{c}_{\eps}(y_{*}))\cap\Gamma=V^{c}_{\eps}(f^{n}(y*))\cap\Gamma=f^{n}(y_{*})$. The continuity of the pre-lamination $V^{c}(x)$ implies that $f^{n_{k}}(y_{*})\to y_{*}$, therefore $y_{*}\in W^{su}(x_{0})$ is a recurrent point. Since the dynamics of $f$ on $\Gamma$ is hyperbolic, the Anosov closing lemma implies there is a sequence of periodic points converging to $y_{*}$ in the relative topology. This implies that for close enough periodic points, they are strongly heteroclinically related.
This finishes the proof.


\begin{thebibliography}{RRR}
\bibitem{avila2008} A. \'{A}vila, On the regularization of conservative maps, {\em Acta Mathematica}, {\bf 205} (2010), 5-18.
\bibitem{avila_bochi2009} A. \'{A}vila, J. Bochi, Nonuniform hyperbolicity, global dominated splittings and generic properties of volume-preserving diffeomorphisms, {\em Transactions AMS} {\bf 364} no. 6 (2012),  2883-2907
\bibitem{baraviera_bonatti2003} A. Baraviera, C. Bonatti, Removing zero Lyapunov exponents, {\em Ergod. Th. \& Dynam. Sys.} {\bf 23} (2003) 1655-1670.
\bibitem{bochi2002} J. Bochi, Genericity of zero Lyapunov exponents, {\em Erg. Th. \& Dyn. Sys.} {\bf 22}, 1667-16960, (2002).
\bibitem{bochi_viana2005} J. Bochi, M. Viana, The Lyapunov exponents of generic volume-preserving and symplectic maps, {\em Ann. Math.} {\bf 161}, (2005) 1423-1485.
\bibitem{bonatti_matheus_viana_wilkinson} C. Bonatti, C. Matheus, M. Viana, A. Wilkinson, Abundance of stable ergodicity, {\em Comment. Math. Helv.} {\bf 79} (2004) 753-757.
\bibitem{gabai_kazez} D. Gabai and W. Kazez, Group negative curvature for 3-manifolds with genuine laminations,
{\em Geom. Topol.},{\bf 2} (1998), 65�77 (electronic).
\bibitem{grin2010} E. Grin, Genericity of diffeomorphisms with zero Lyapunov exponents almost everywhere, {\em Msc. Thesis}, Montevideo, 2010.
\bibitem{haefliger} A. Haefliger, Vari\'et\'es feuillet\'ees, {\it Topologia Differenziale
(Centro Int. Mat. Estivo, 1 deg Ciclo, Urbino (1962)). Lezione 2 Edizioni Cremonese, Rome 367-397 (1962)}
\bibitem{hector_hirsch} G. Hector and U. Hirsch, {\em Introduction to the Geometry of Foliations: Part B, Foliations
of Codimension one},  Second edition, Aspects of Mathematics, Vieweg, Braunschweig/
Wiesbaden, 1987.
\bibitem{HPS} M. Hirsch, C. Pugh, M. Shub, {\em Invariant manifolds}, Lect. Notes Math. {\bf 583}, Springer Verlag 1977.
\bibitem{journe} J.-L. Journ\'e, A regularity lemma for functions of several variables, Rev. Mat. Iberoamericana 4 (1988), 187Ð193.
\bibitem{manhe1982} R. Ma\~n\'e, An ergodic closing lemma, {\em Ann. of Math.}, {\bf 116}, 3
(1982) 503-540.
\bibitem{manhe1983} R. Ma\~{n}\'{e}, Oseledec's theorem from the generic viewpoint, {\em Proc. Internat. Congress of Mathematicians} Warsaw, vol. 2 (1983) 1269-1276.
\bibitem{oseledec} V. Oseledec, A multiplicative ergodic theorem: Lyapunov characteristic numbers for dynamical systems, {\em Trans. Moscow Math. Soc.} {\bf 19} (1968), 197-231.
\bibitem{oxtoby_ulam} J. Oxtoby, S. Ulam, Measure-preserving homeomorphisms and metrical transitivity, {\em Ann. of Math.} {\bf 42} (1941) 874-920.
\bibitem{pesin} Ya. Pesin, Characteristic Lyapunov exponents and smooth ergodic theory, Russian Math. Surveys {\bf 32} (1977), 55Ð114. MR 0466791. 
\bibitem{rhrhu2008} F. Rodriguez Hertz, M. Rodriguez Hertz, R. Ures, Accessibility and stable ergodicity for partially hyperbolic diffeomorphisms with 1D-center bundle, {\em Invent. Math.} {\bf 172} (2008) 353-381.
\bibitem{rhrhu2008_2} F. Rodriguez Hertz, M. Rodriguez Hertz, R. Ures, Partial hyperbolicity and ergodicity in dimension 3,  {\em Journal of Modern Dynamics} {\bf 2}, No.2 (2008) 187-208
\bibitem{rhrhtu2009} F. Rodriguez Hertz, M. Rodriguez Hertz, A. Tahzibi, R. Ures, New criteria for ergodicity and non-uniform hyperbolicity, {\em Duke Math. Journal} {\bf 160}, 3, (2011) 599-629.
\bibitem{zhang} Zhang, P. Partially hyperbolic sets with positive measure and $ACIP$ for partially hyperbolic systems, {\em Disc. Cont. Dyn. Sys.} {\bf 32}, 4 (2012) 1435-1447.
\end{thebibliography}
\end{document}